\documentclass{svjour3}                     % onecolumn (standard format)
\smartqed  % flush right qed marks, e.g. at end of proof
\usepackage{bbm}
\usepackage{mathrsfs}
\usepackage{amsfonts}
\usepackage{stmaryrd}
\usepackage{graphicx}
\usepackage{subfigure}

\usepackage{bm}% bold math
\usepackage{amsmath}
\usepackage{amssymb}

\usepackage{amstext}
\usepackage{amsbsy}
\usepackage{setspace}
\usepackage{algorithm}
\usepackage{algorithmic}
\usepackage[figuresright]{rotating}
\usepackage{epstopdf}
\usepackage{makecell}
\usepackage{color}
\usepackage{booktabs}

\numberwithin{equation}{section}
\numberwithin{theorem}{section}
\numberwithin{lemma}{section}
\numberwithin{remark}{section}
\usepackage{mathptmx}

\newtheorem{ass}{Assumption}

\usepackage{mathptmx}      % use Times fonts if available on your TeX system
\usepackage[justification=centering]{caption}
\date{}
\begin{document}

\title{Strong approximation for fractional wave equation forced by fractional Brownian motion with Hurst parameter $H\in(0,\frac{1}{2})$}

%\titlerunning{Short form of title}        % if too long for running head

\author{Xing Liu
}

%\authorrunning{Short form of author list} % if too long for running head

\institute{
Xing Liu \at
School of Mathematics and Economics, Bigdata Modeling and Intelligent Computing research institute, Hubei University of Education, Wuhan 430205, People's Republic of China. \email{2718826413@qq.com}
             % Tel.: +123-45-678910\\
%              Fax: +123-45-678910\\
             % \email{wsqlzu@gmail.com}
 %  \emph{Present address:} of F. Author  %  if needed
%\and
%Weihua Deng\at School of Mathematics and Statistics, Gansu Key Laboratory of Applied Mathematics and Complex
%Systems, Lanzhou University, Lanzhou 730000, People's Republic of China.
%\email{dengwh@lzu.edu.cn}
}

\maketitle

\begin{abstract}	
We consider the time discretization of fractional stochastic wave equation with Gaussian noise, which is negatively correlated. Major obstacles to design and analyze time discretization of stochastic wave equation come from the approximation of stochastic convolution with respect to fractional Brownian motion. Firstly, we discuss the smoothing properties of stochastic convolution by using integration by parts and covariance function of fractional Brownian motion. Then the regularity estimates of the mild solution of fractional stochastic wave equation are obtained. Next, we design the time discretization of stochastic convolution by integration by parts. Combining stochastic trigonometric method and approximation of stochastic convolution, the time discretization of stochastic wave equation is achieved. We derive the error estimates of the time discretization. Under certain assumptions, the strong convergence rate of the numerical scheme proposed in this paper can reach $\frac{1}{2}+H$. Finally, the convergence rate and computational efficiency of the numerical scheme are illustrated by numerical experiments.
\\
\\
%\noindent{\bf AMS} 26A33, 35R11, 65M60, 65M12.

% \PACS{PACS code1 \and PACS code2 \and more}
% \subclass{MSC code1 \and MSC code2 \and more}
\end{abstract}

\section{Introduction}
Stochastic partial differential equations (SPDEs) can realistically simulate many phenomena in physical scientific and engineering applications; see \cite{BrunedGabriel,GaoMar,RyserNig}. The theoretical and numerical studies of SPDEs have received much attention \cite{Brehier,BrunedChand,Jentzen,Liu1,SongZh}. While most works of stochastic model in fractional Brownian motion (FBM) are carried out for Hurst parameter $H\in[\frac{1}{2},1)$ (\cite{Anh,Hong,Kamrani,LiWang,Salins}), a FBM with $H\in(0,\frac{1}{2})$ might be more reasonable to model sequences with intermittency and anti-persistence, such as visual feedback effects in biology \cite{Boudrahem} and option prices in market practice \cite{Bayer,Gatheral,Simonsen}. Sometimes the stochastic disturbance of the source in PDEs is anti-correlated, and it can be well modeled by the FBM with $H\in(0,\frac{1}{2})$. In this paper, the following stochastic wave equation (SWE) is considered
 \begin{equation}\label{4eq:1.1}
\mathrm{d}\dot{u}(x,t)=-A^\alpha u(x,t)\mathrm{d}t+f\left(u(x,t)\right)\mathrm{d}t+\mathrm{d}B^Q_H(x,t),\quad (x,t)\in \Omega\times [0,T]
\end{equation}
with homogenous Dirichlet boundary condition
\begin{eqnarray*}
\begin{split}
 u(x,0)=u_0(x), \ \dot{u}(x,0)=v_0(x),\quad x\in \Omega
 \end{split}
  \end{eqnarray*}
 and the initial condition
  \begin{eqnarray*}
\begin{split}
  u(x,t)=0, \quad x\in \partial\Omega\times [0,T].
 \end{split}
  \end{eqnarray*}
Here $A^\alpha$ represents the spectral fractional Laplacian with $A=-\mathrm{\Delta}$ and $\alpha \in (0,1)$. The external fluctuation of \eqref{4eq:1.1} is characterized by FBM $B^Q_H(x,t)$ with $H \in (0,\frac{1}{2})$, $Q$ is a bounded covariance operator. Let $\Omega\subseteq \mathbb{R}^d~(d=1,2,3)$ denote a Lispschitz domain.

For $H\in[\frac{1}{2},1)$, the approximation of \eqref{4eq:1.1} has been well studied. The discrete schemes of \cite{WangGan} strongly converge with order $1-\epsilon$ in time for arbitrarily small $\epsilon> 0$. By using an explicit stochastic position Verlet scheme \cite{Banjai}, linear convergence was obtained under certain assumptions. Using It\^o isometry, authors of \cite{LiuDeng} constructed a modified stochastic trigonometric method, which achieved superlinear convergence in time. The computation of SWE with $H \in (0,\frac{1}{2})$ has so far received little attention. In the case of $H \in (0,\frac{1}{2})$, the above methods are not suitable for dealing with fractional noise of \eqref{4eq:1.1}, because the stochastic process $B^Q_H(x,t)$ is less regular in time. The main challenges lie in time discretization of stochastic convolution with respect to FBM and the corresponding error estimate. Furthermore, we noticed that $f$ usually satisfies Lipschitz continuous condition and linear growth condition in the existing work. In this paper, the linear growth condition is replaced by a weaker condition.

The objective of this paper is to derive an effective numerical scheme of \eqref{4eq:1.1} in time and establish error estimates. The main process is as follows. Firstly, by combining integration by parts, the covariance function of FBM and trigonometric identity, we establish the regularity of mild solution of \eqref{4eq:1.1} including space regularity and time H\"older continuity. After that, we adapt the stochastic trigonometric method to discretize the problem \eqref{4eq:1.1} in time. In order to realize the stochastic trigonometric method, a time discretization of stochastic convolution is designed by integration by parts. Finally, the error estimate of time discretization is derived by using time H\"older continuity of mild solution. In addition, under appropriate conditions of noise the strong convergence rate of order $H+\frac{1}{2}$ is proved by combining integration by parts and covariance function of FBM.

The outline of the rest of this paper is as follows. In section \ref{4sec:2}, we use a SPDE to define the mild solution of SWE \eqref{4eq:1.1}. We present the regularity estimates of the mild solution in section \ref{4sec:3}. In section \ref{4sec:4}, we use integration by parts to reformulate stochastic convolution with respect to FBM into an integral formula, of which it is easy to obtain time discretization of stochastic convolution and corresponding error estimate. Then the discretization of \eqref{4eq:1.1} in time is designed, and the strong convergence order of discretization is derived. We present numerical experiments to confirm the strong convergence order of our numerical scheme in section \ref{4sec:5}. At last, some conclusions are given in section \ref{4sec:6}.

\section{Representation of mild solution} \label{4sec:2}
In this section, we define the mild solution of SWE \eqref{4eq:1.1}. In addition, in order to obtain regularity estimates of the mild solution, several assumptions are introduced.

First, we introduce the norm space $\mathbb{H}^\nu$
\begin{equation*}
\mathbb{H}^\nu=\left\{u:\Omega\to \mathbb{R}\bigg|\ \left\|A^{\frac{\nu}{2}}u\right\|=\left(\int_\Omega\left|A^{\frac{\nu}{2}}u\right|^2\mathrm{d}x\right)^{\frac{1}{2}}<\infty,~~\nu\in\mathbb{R}\right\}.
\end{equation*}
Let $\left\{e_{i}\right\}_{i\in\mathbb{N}}$ be an orthonormal basis of $\mathbb{H}^0$. For $u\in \mathbb{H}^0$, we have
\begin{equation*}
u=\sum^\infty_{i=1}\left\langle u,e_i\right\rangle e_i,
\end{equation*}
where $\left\langle\cdot, \cdot\right\rangle$ denotes the inner product of $\mathbb{H}^0$. Then
\begin{equation*}
 \|u\|^2=\sum^\infty_{i=1}\left\langle u,e_i\right\rangle^2.
\end{equation*}
As $u\in \mathbb{H}^\nu$, using the orthonormal eigenpairs of $A$, we have
\begin{equation*}
A^{\frac{\nu}{2}}u=\sum^\infty_{i=1}\lambda_i^{\frac{\nu}{2}}\left\langle u,\phi_i\right\rangle\phi_i
\end{equation*}
and
\begin{equation*}
 \left\|A^{\frac{\nu}{2}}u\right\|^2=\sum^\infty_{i=1}\lambda_i^{\nu}\left\langle u,\phi_i\right\rangle^2<\infty.
\end{equation*}
%The space $\dot{U}^{\nu}$ is defined as
%\begin{equation*}
%\dot{U}^{\nu}=\left\{u\in U \bigg|~~\|A^{\frac{\nu}{2}}u\|^2<\infty,~~\nu\in\mathbb{R}\right\}.
%\end{equation*}
The $B^Q_H(x,t)$ in \eqref{4eq:1.1} is represented as
\begin{equation*}
B^Q_H(x,t)=\sum^\infty_{i=1}\sqrt{q}_{i}\xi_H^{i}(t)\phi_{i}(x),
\end{equation*}
where $\{\xi_H^{i}(t)\}_{i\in\mathbb{N}}$ are mutually independent real-valued one-dimensional FBM. Its covariance function is given by
\begin{equation*}
\mathrm{E}\left[\xi_H^{i}(t)\xi_H^{i}(s)\right]=\frac{1}{2}\left(t^{2H}+s^{2H}-|t-s|^{2H}\right),
\end{equation*}
for any $s, t\in \mathbb{R}^+$. As $H \in (0,\frac{1}{2})$, the increments of FBM are negatively correlated \cite{AlosMazet1}. Let $u(t)=u(x,t)$ and $B^Q_H(t)=B^Q_H(x,t)$. Next we consider the mild solution of the following SPDE
\begin{equation}\label{4eq:2.1}
\mathrm{d}\left[\begin{split}u(t)\\ \dot{u}(t)\end{split}\right]=\left[\begin{split}0~~~~~I\\ -A^\alpha~~0\end{split}\right]\left[\begin{split}u(t)\\ \dot{u}(t)\end{split}\right]\mathrm{d}t+\left[\begin{split}0~~~~\\ f\left(u(t)\right)\end{split}\right]\mathrm{d}t+\left[\begin{split}0\\ I\end{split}\right]\mathrm{d}B^Q_{H}(t).
\end{equation}
For $t\ge0$, we give the definition of a continuous semigroup $\mathscr{L}(t)$.
\begin{equation}\label{4eq:2.2}
\mathscr{L}(t)=\left[\begin{split}\mathscr{C}(t)~~~~~A^{-\frac{\alpha}{2}}\mathscr{S}(t)\\ -A^{\frac{\alpha}{2}}\mathscr{S}(t)~~~~~~\mathscr{C}(t)\end{split}\right].
\end{equation}
Here $\mathscr{C}(t)=\cos\left(A^{\frac{\alpha}{2}}t\right)$ and $\mathscr{S}(t)=\sin\left(A^{\frac{\alpha}{2}}t\right)$. Using the orthonormal eigenpairs of $A$, we have the following expansions
\begin{equation*}
\cos\left(A^{\frac{\alpha}{2}}t\right)u(t)=\sum^\infty_{i=0}\cos\left(\lambda_i^{\frac{\alpha}{2}}t\right)\left\langle u(t),\phi_i\right\rangle\phi_i
\end{equation*}
and
\begin{equation*}
\sin\left(A^{\frac{\alpha}{2}}t\right)u(t)=\sum^\infty_{i=0}\sin\left(\lambda_i^{\frac{\alpha}{2}}t\right)\left\langle u(t),\phi_i\right\rangle\phi_i.
\end{equation*}
Then using constant variation method to solve \eqref{4eq:2.1}, we have
\begin{equation}\label{4eq:2.3}
\left[\begin{split}u(t)\\ \dot{u}(t)\end{split}\right]=\mathscr{L}(t-s)\left[\begin{split}u(s)\\ \dot{u}(s)\end{split}\right]+\int^t_s\mathscr{L}(t-r)\left[\begin{split}0~~~~\\ f\left(u(r)\right)\end{split}\right]\mathrm{d}r+\int^t_s\mathscr{L}(t-r)\left[\begin{split}0\\ I\end{split}\right]\mathrm{d}B^Q_{H}(r).
\end{equation}
Combining \eqref{4eq:2.2} and the rule of matrix multiplication, \eqref{4eq:2.3} can be written as
\begin{equation}\label{4eq:2.4}
\begin{split}
u(t)=&\mathscr{C}(t-s)u(s)+A^{-\frac{\alpha}{2}}\mathscr{S}(t-s)\dot{u}(s)\\
&~~~~+\int^t_sA^{-\frac{\alpha}{2}}\mathscr{S}(t-r)f\left(u(r)\right)\mathrm{d}r+\int^t_sA^{-\frac{\alpha}{2}}\mathscr{S}(t-r)\mathrm{d}B^Q_H(r),\\
\dot{u}(t)=&-A^{\frac{\alpha}{2}}\mathscr{S}(t-s)u(s)+\mathscr{C}(t-s)\dot{u}(s)\\
&~~~~+\int^t_s\mathscr{C}(t-r)f\left(u(r)\right)\mathrm{d}r+\int^t_s\mathscr{C}(t-r)\mathrm{d}B^Q_H(r).
\end{split}
\end{equation}
The mild solution of SWE \eqref{4eq:1.1} can be defined by \eqref{4eq:2.4}. The regularity of $u(t)$ relies on smoothing properties of $\int^t_sA^{-\frac{\alpha}{2}}\mathscr{S}(t-r)\mathrm{d}B^Q_H(r)$ in the space $L^2(\Omega,\mathbb{H}^\nu)$. For $\nu\in \mathbb{R}$, we equip this space $L^2(\Omega,\mathbb{H}^\nu)$ with the norm
\begin{equation*}
\left\|u\right\|_{L^2\left(\Omega,\mathbb{H}^\nu\right)}=\left(\mathrm{E}\left[\left\|A^{\frac{\nu}{2}}u\right\|^2\right]\right)^{\frac{1}{2}}.
\end{equation*}
In fact, as $0<\nu<\frac{1}{2}$, the regularity of $u(t)$ can also be described by classical fractional Sobolev spaces. Let $W^{v,2}$ denote the fractional Sobolev spaces
\begin{equation*}
W^{v,2}=\left\{u\in U | \ \|u\|_{W^{v,2}}<\infty\right\},
\end{equation*}
where
\begin{equation*}
\|u\|^2_{W^{v,2}}=\left\|u\right\|^2+\int_\Omega\int_\Omega\frac{\left|u(x)-u(y)\right|^2}{\left|x-y\right|^{d+2\nu}}\mathrm{d}x\mathrm{d}y.
\end{equation*}
For $0<\nu<\frac{1}{2}$, $W^{v,2}$ is identical to $\mathbb{H}^\nu$ \cite{McLean}, which implies
\begin{equation}\label{4eq:2.5}
\left\|u\right\|^2_{L^2\left(\Omega,\mathbb{H}^\nu\right)}\lesssim\mathrm{E}\left[\|u\|^2_{W^{v,2}}\right]\lesssim\left\|u\right\|^2_{L^2\left(\Omega,\mathbb{H}^\nu\right)},~~\mathrm{for}~~u\in L^2\left(\Omega,\mathbb{H}^\nu\right).
\end{equation}
In order to establish the well-posedness of \eqref{4eq:1.1}, we require the stochastic process $B^Q_H(x,t)$ to satisfy the following assumption.
\begin{ass}\label{as:2.1}
 The covariance operator $Q$ be a self~adjoint, nonnegative linear operator on $\mathbb{H}^0$ and $Q\phi_{i}(x)=q_{i}\phi_{i}(x)$. Here $q_{i}$ is a non-negative real number. For $-\frac{\alpha}{4}-\frac{H\alpha}{2}<\rho<\frac{\alpha+1}{4}-\frac{H\alpha}{2}$, we assume that $A^{\rho}Q^{\frac{1}{2}}$ is a bounded operator on $\mathbb{H}^0$, i.e.,
\begin{equation*}
\sum_{i=1}^{\infty}\left\|A^{\rho}Q^{\frac{1}{2}}\phi_{i}(x)\right\|^2\lesssim 1.
\end{equation*}
\end{ass}

\begin{ass}\label{as:2.2}
 For $t,~s\in\mathbb{R}$, the function $f: \mathbb{R}\to\mathbb{R}$ satisfies
\begin{equation*}
|f(t)-f(s)|\lesssim |t-s|
\end{equation*}
and
\begin{equation*}
|f(0)|\lesssim 1.
\end{equation*}
\end{ass}

\section{Regularity of the Solution } \label{4sec:3}
In this section, we establish the regularity estimates of the mild solution for SWE \eqref{4eq:1.1}.

Before establishing the regularity estimates of \eqref{4eq:1.1}, we introduce the definition of stochastic integral with respect to the FBM with $H \in (0,\frac{1}{2})$ \cite{Bardina}.
\begin{equation*}
\begin{split}
\int^t_sg(r)\mathrm{d}\xi_H(r)&=\lim_{\delta t\to0}\sum^{n-1}_{j=0}g_j\left(\xi_H(t_{j+1})-\xi_H(t_j)\right),
\end{split}
\end{equation*}
where $\delta t=\max_{0\le j\le n-1}(t_{j+1}-t_j)$, and $s=t_0< t_1<\dots<t_{n}=t$. Then we can get the following integral by using integration by parts.
\begin{equation}\label{4eq:3.5}
\int^t_sg(r)\mathrm{d}\xi_H(r)=g(t)\xi_H(t)-g(s)\xi_H(s)-\int^t_sg'(r)\xi_H(r)\mathrm{d}r.
\end{equation}

Next, we apply \eqref{4eq:3.5} and the covariance function of FBM to derive the regularity estimates of stochastic convolution.

\begin{proposition}\label{4prop:1}
Let $\kappa=\frac{\alpha}{2}+H\alpha+2\rho$ and $\kappa>0$, assumption \ref{as:2.1} holds, then
\begin{equation}\label{4eq:3.5-1}
\begin{split}
\mathrm{E}\left[\left\|\int^t_sA^{\frac{\kappa-\alpha}{2}}\mathscr{S}(t-r)\mathrm{d}B^Q_H(r)\right\|^2\right]+\mathrm{E}\left[\left\|\int^t_sA^{\frac{\kappa-\alpha}{2}}\mathscr{C}(t-r)\mathrm{d}B^Q_H(r)\right\|^2\right]\lesssim1.
\end{split}
\end{equation}
Furthermore
\begin{equation}\label{4eq:3.5-2}
\begin{split}
\mathrm{E}\left[\left\|A^{\frac{-\alpha}{2}}\int^t_s\mathscr{S}(t-r)\mathrm{d}B^Q_H(r)\right\|^2\right]\lesssim(t-s)^{\min\left\{\frac{2\kappa}{\alpha},2\right\}}.
\end{split}
\end{equation}

\end{proposition}
\begin{proof}
Term-by-term integration by parts for $\int^t_sA^{-\frac{\alpha}{2}}\mathscr{S}(t-r)\mathrm{d}B^Q_H(r)$ and $\int^t_s\mathscr{S}(t-r)\mathrm{d}B^Q_H(r)$ shows that
\begin{equation}\label{4eq:3.6}
\begin{split}
\int^t_sA^{\frac{\kappa-\alpha}{2}}\mathscr{S}(t-r)\mathrm{d}B^Q_H(r)=-A^{\frac{\kappa-\alpha}{2}}\mathscr{S}(t-s)B^Q_H(s)+\int^t_sA^{\frac{\kappa}{2}}\mathscr{C}(t-r)B^Q_H(r)\mathrm{d}r
\end{split}
\end{equation}
and
\begin{equation}\label{4eq:3.6-1}
\begin{split}
\int^t_s\mathscr{C}(t-r)\mathrm{d}B^Q_H(r)=B^Q_H(t)-\mathscr{C}(t-s)B^Q_H(s)-\int^t_sA^{\frac{\alpha}{2}}\mathscr{S}(t-r)B^Q_H(r)\mathrm{d}r.
\end{split}
\end{equation}
Then using \eqref{4eq:3.6}, we get
\begin{equation}\label{4eq:3.7}
\begin{split}
&\mathrm{E}\left[\left\|\int^t_sA^{\frac{\kappa-\alpha}{2}}\mathscr{S}(t-r)\mathrm{d}B^Q_H(r)\right\|^2\right]\\
&\lesssim\mathrm{E}\left[\left\|A^{\frac{\kappa-\alpha}{2}}\sin\left(A^{\frac{\alpha}{2}}(t-s)\right)B^Q_H(s)\right\|^2\right]+\mathrm{E}\left[\left\|\int^t_sA^{\frac{\kappa}{2}}\cos\left(A^{\frac{\alpha}{2}}(t-r)\right)B^Q_H(r)\mathrm{d}r\right\|^2\right]\\
&=\sum^\infty_{i=1}\lambda_i^{\kappa-\alpha}q_{i}\sin^2\left(\lambda_i^{\frac{\alpha}{2}}(t-s)\right) s^{2H}\\
&~~~~+\sum^\infty_{i=1}\lambda_i^{\kappa}q_{i}\int^t_s\int^t_s \cos\left(\lambda_i^{\frac{\alpha}{2}}(t-r)\right)\cos\left(\lambda_i^{\frac{\alpha}{2}}(t-r_1)\right)\frac{r^{2H}+r_1^{2H}}{2}\mathrm{d}r\mathrm{d}r_1\\
&~~~~-\sum^\infty_{i=1}\lambda_i^{\kappa}q_{i}\int^t_s\int^t_s \cos\left(\lambda_i^{\frac{\alpha}{2}}(t-r)\right)\cos\left(\lambda_i^{\frac{\alpha}{2}}(t-r_1)\right)\frac{|r-r_1|^{2H}}{2}\mathrm{d}r\mathrm{d}r_1\\
&=\sum^\infty_{i=1}\lambda_i^{\kappa-\alpha}q_i\sin^2\left(\lambda_i^{\frac{\alpha}{2}}(t-s)\right) s^{2H}\\
&~~~~+\sum^\infty_{i=1}\lambda_i^{\kappa-\frac{\alpha}{2}}q_i\sin\left(\lambda_i^{\frac{\alpha}{2}}(t-s)\right)\int^t_s \cos\left(\lambda_i^{\frac{\alpha}{2}}(t-r)\right)r^{2H}\mathrm{d}r\\
&~~~~-\sum^\infty_{i=1}\lambda_i^{\kappa}q_i\int^t_s\int^r_{s} \cos\left(\lambda_i^{\frac{\alpha}{2}}(t-r)\right)\cos\left(\lambda_i^{\frac{\alpha}{2}}(t-r_1)\right)(r-r_1)^{2H}\mathrm{d}r_1\mathrm{d}r\\
&\lesssim\sum^\infty_{i=1}\lambda_i^{\kappa-\alpha}q_is^{2H}+2H\sum^\infty_{i=1}\lambda_i^{\kappa-\alpha}q_i\sin\left(\lambda_i^{\frac{\alpha}{2}}(t-s)\right)\int^t_s \sin\left(\lambda_i^{\frac{\alpha}{2}}(t-r)\right)r^{2H-1}\mathrm{d}r\\
&~~~~-\sum^\infty_{i=1}\lambda_i^{\kappa}q_i\int^t_s\int^r_{s} \cos\left(\lambda_i^{\frac{\alpha}{2}}(t-r)\right)\cos\left(\lambda_i^{\frac{\alpha}{2}}(t-r_1)\right)(r-r_1)^{2H}\mathrm{d}r_1\mathrm{d}r.
\end{split}
\end{equation}
The first inequality is due to thetrigonometric inequality and the covariance function of FBM. The second equality is due to the integration by parts. Similarly, we have
\begin{equation}\label{4eq:3.7-1}
\begin{split}
&\mathrm{E}\left[\left\|\int^t_sA^{\frac{\kappa-\alpha}{2}}\mathscr{C}(t-r)\mathrm{d}B^Q_H(r)\right\|^2\right]\\
&\lesssim\sum^\infty_{i=1}\lambda_i^{\kappa-\alpha}q_{i}t^{2H}\\
&~~~~+\sum^\infty_{i=1}\lambda_i^{\kappa}q_{i}\int^t_s\int^t_s \sin\left(\lambda_i^{\frac{\alpha}{2}}(t-r)\right)\sin\left(\lambda_i^{\frac{\alpha}{2}}(t-r_1)\right)\frac{r^{2H}+r_1^{2H}-|r-r_1|^{2H}}{2}\mathrm{d}r\mathrm{d}r_1\\
&=\sum^\infty_{i=1}\lambda_i^{\kappa-\alpha}q_it^{2H}+\sum^\infty_{i=1}\lambda_i^{\kappa-\frac{\alpha}{2}}q_i\left(1-\cos\left(\lambda_i^{\frac{\alpha}{2}}(t-s)\right)\right)\int^t_s \sin\left(\lambda_i^{\frac{\alpha}{2}}(t-r)\right)r^{2H}\mathrm{d}r\\
&~~~~-\sum^\infty_{i=1}\lambda_i^{\kappa}q_i\int^t_s\int^r_{s} \sin\left(\lambda_i^{\frac{\alpha}{2}}(t-r)\right)\sin\left(\lambda_i^{\frac{\alpha}{2}}(t-r_1)\right)(r-r_1)^{2H}\mathrm{d}r_1\mathrm{d}r\\
&\lesssim\sum^\infty_{i=1}\lambda_i^{\kappa-\alpha}q_it^{2H}-\sum^\infty_{i=1}\lambda_i^{\kappa-\alpha}q_i\left(1-\cos\left(\lambda_i^{\frac{\alpha}{2}}(t-s)\right)\right)\int^t_s \cos\left(\lambda_i^{\frac{\alpha}{2}}(t-r)\right)r^{2H-1}\mathrm{d}r\\
&~~~~-\sum^\infty_{i=1}\lambda_i^{\kappa}q_i\int^t_s\int^r_{s} \sin\left(\lambda_i^{\frac{\alpha}{2}}(t-r)\right)\sin\left(\lambda_i^{\frac{\alpha}{2}}(t-r_1)\right)(r-r_1)^{2H}\mathrm{d}r_1\mathrm{d}r.
\end{split}
\end{equation}
Combining \eqref{4eq:3.7} and \eqref{4eq:3.7-1} leads to
\begin{equation*}
\begin{split}
&\mathrm{E}\left[\left\|\int^t_sA^{\frac{\kappa-\alpha}{2}}\mathscr{S}(t-r)\mathrm{d}B^Q_H(r)\right\|^2\right]+\mathrm{E}\left[\left\|\int^t_sA^{\frac{\kappa-\alpha}{2}}\mathscr{C}(t-r)\mathrm{d}B^Q_H(r)\right\|^2\right]\\
&\lesssim\sum^\infty_{i=1}\lambda_i^{\kappa-\alpha}q_it^{2H}\\
&~~~~-\sum^\infty_{i=1}\lambda_i^{\kappa}q_i\int^t_s\int^r_{s} \cos\left(\lambda_i^{\frac{\alpha}{2}}(t-r)\right)\cos\left(\lambda_i^{\frac{\alpha}{2}}(t-r_1)\right)(r-r_1)^{2H}\mathrm{d}r_1\mathrm{d}r\\
&~~~~-\sum^\infty_{i=1}\lambda_i^{\kappa}q_i\int^t_s\int^r_{s} \sin\left(\lambda_i^{\frac{\alpha}{2}}(t-r)\right)\sin\left(\lambda_i^{\frac{\alpha}{2}}(t-r_1)\right)(r-r_1)^{2H}\mathrm{d}r_1\mathrm{d}r\\
&\lesssim\sum^\infty_{i=1}\lambda_i^{\kappa-\alpha}q_{i}t^{2H}-\sum^\infty_{i=1}\lambda_i^{\kappa}q_i\int^t_s\int^r_{s}\cos\left(\lambda_i^{\frac{\alpha}{2}}(r-r_1)\right)(r-r_1)^{2H}\mathrm{d}r_1\mathrm{d}r\\
&=\sum^\infty_{i=1}\lambda_i^{\kappa-\alpha}q_{i}t^{2H}-\sum^\infty_{i=1}\lambda_i^{\kappa-\frac{\alpha}{2}}q_i\int^t_s\sin\left(\lambda_i^{\frac{\alpha}{2}}(r-s)\right)(r-s)^{2H}\mathrm{d}r\\
&~~~~+2H\sum^\infty_{i=1}\lambda_i^{\kappa-\frac{\alpha}{2}}q_i\int^t_s\int^r_{s}\sin\left(\lambda_i^{\frac{\alpha}{2}}(r-r_1)\right)(r-r_1)^{2H-1}\mathrm{d}r_1\mathrm{d}r\\
&\lesssim\sum^\infty_{i=1}\lambda_i^{\kappa-\alpha}q_{i}t^{2H}+\sum^\infty_{i=1}\lambda_i^{\kappa-\alpha}q_i\cos\left(\lambda_i^{\frac{\alpha}{2}}(t-s)\right)(t-s)^{2H}\\
&~~~~-2H\sum^\infty_{i=1}\lambda_i^{\kappa-\alpha}q_i\int^t_s\cos\left(\lambda_i^{\frac{\alpha}{2}}(r-s)\right)(r-s)^{2H-1}\mathrm{d}r\\
&~~~~+2H\sum^\infty_{i=1}\lambda_i^{\kappa-\frac{\alpha}{2}}q_i\int^t_s\int^{\infty}_{0}\sin\left(\lambda_i^{\frac{\alpha}{2}}l\right)l^{2H-1}\mathrm{d}l\mathrm{d}r\\
&~~~~-2H\sum^\infty_{i=1}\lambda_i^{\kappa-\frac{\alpha}{2}}q_i\int^t_s\int^{\infty}_{r-s}\sin\left(\lambda_i^{\frac{\alpha}{2}}l\right)l^{2H-1}\mathrm{d}l\mathrm{d}r\\
&\lesssim\sum^\infty_{i=1}\lambda_i^{\kappa-\alpha}q_{i}t^{2H}+2H\sum^\infty_{i=1}\lambda_i^{\kappa-\frac{\alpha}{2}}q_i\int^t_s\lambda_i^{-H\alpha}\Gamma(2H)\sin(H\alpha)\mathrm{d}r\\
&~~~~-2H\sum^\infty_{i=1}\lambda_i^{\kappa-\alpha}q_i\int^t_s\cos\left(\lambda_i^{\frac{\alpha}{2}}(r-s)\right)(r-s)^{2H-1}\mathrm{d}r\\
&~~~~-2H(2H-1)\sum^\infty_{i=1}\lambda_i^{\kappa-\alpha}q_i\int^t_s\int^{\infty}_{r-s}\cos\left(\lambda_i^{\frac{\alpha}{2}}l\right)l^{2H-2}\mathrm{d}l\mathrm{d}r\\
&\lesssim\sum^\infty_{i=1}\lambda_i^{\kappa-\frac{\alpha}{2}-H\alpha}q_i\\
&=\sum_{i=1}^{\infty}\left\|A^{\rho}Q^{\frac{1}{2}}\phi_{i}(x)\right\|^2\\
&\lesssim1.
\end{split}
\end{equation*}
In the fourth inequality, we use the following equation. Thanks to assumption \ref{as:2.1}, we arrive at the last inequality.
\begin{equation*}
\int^\infty_0\sin\left(al\right)l^{\theta-1}\mathrm{d}l=\frac{\Gamma(\theta)}{a^\theta}\sin(\frac{\theta\pi}{2}),
\end{equation*}
where $0<\theta<1$, $a>0$.

Next, we turn to the H\"older regularity estimates of stochastic convolution. For $0<\kappa\le H\alpha$, we obtain the following estimates by using \eqref{4eq:3.7}.
\begin{equation*}
\begin{split}
&\mathrm{E}\left[\left\|\int^t_sA^{\frac{-\alpha}{2}}\mathscr{S}(t-r)\mathrm{d}B^Q_H(r)\right\|^2\right]\\
&\lesssim\sum^\infty_{i=1}\lambda_i^{-\alpha}q_i\sin^2\left(\lambda_i^{\frac{\alpha}{2}}(t-s)\right) s^{2H}\\
&~~~~+\sum^\infty_{i=1}\lambda_i^{-\frac{\alpha}{2}}q_i\sin\left(\lambda_i^{\frac{\alpha}{2}}(t-s)\right)\int^t_s \cos\left(\lambda_i^{\frac{\alpha}{2}}(t-r)\right)r^{2H}\mathrm{d}r\\
&~~~~-\sum^\infty_{i=1}q_i\int^t_s\int^r_{s} \cos\left(\lambda_i^{\frac{\alpha}{2}}(t-r)\right)\cos\left(\lambda_i^{\frac{\alpha}{2}}(t-r_1)\right)(r-r_1)^{2H}\mathrm{d}r_1\mathrm{d}r\\
&\lesssim2\sum^\infty_{i=1}\lambda_i^{-\alpha}q_{i}\sin^2\left(\lambda_i^{\frac{\alpha}{2}}(t-s)\right)s^{2H}\\
&~~~~+2H\sum^\infty_{i=1}\lambda_i^{-\alpha}q_{i}\sin\left(\lambda_i^{\frac{\alpha}{2}}(t-s)\right)\int^t_s\sin\left(\lambda_i^{\frac{\alpha}{2}}(t-r)\right)r^{2H-1}\mathrm{d}r\\
&~~~~-2H\sum^\infty_{i=1}\lambda_i^{-\frac{\alpha}{2}}q_{i}\int^t_s\int^r_{s} \sin\left(\lambda_i^{\frac{\alpha}{2}}(t-r)\right)\cos\left(\lambda_i^{\frac{\alpha}{2}}(t-r_1)\right)(r-r_1)^{2H-1}\mathrm{d}r_1\mathrm{d}r\\
&\lesssim(t-s)^{\frac{2\kappa}{\alpha}}\sum^\infty_{i=1}\lambda_i^{\kappa-\alpha}q_{i}+\sum^\infty_{i=1}\lambda_i^{-\alpha}q_{i}\left(\lambda_i^{\frac{\alpha}{2}}(t-s)\right)^{\frac{\kappa}{\alpha}}\int^t_s\left(\lambda_i^{\frac{\alpha}{2}}(t-r)\right)^{\frac{\kappa}{\alpha}}r^{2H-1}\mathrm{d}r\\
&~~~~+2H\sum^\infty_{i=1}\lambda_i^{-\alpha}q_{i}\int^t_s \sin\left(\lambda_i^{\frac{\alpha}{2}}(t-r)\right)\left(\sin\left(\lambda_i^{\frac{\alpha}{2}}(t-s)\right)-\sin\left(\lambda_i^{\frac{\alpha}{2}}(t-r)\right)\right)(r-s)^{2H-1}\mathrm{d}r\\
&~~~~+2H(2H-1)\sum^\infty_{i=1}\lambda_i^{-\alpha}q_{i}\int^t_s\int^r_s \sin\left(\lambda_i^{\frac{\alpha}{2}}(t-r)\right)\\
&~~~~\times\left(\sin\left(\lambda_i^{\frac{\alpha}{2}}(t-r_1)\right)-\sin\left(\lambda_i^{\frac{\alpha}{2}}(t-r)\right)\right)(r-r_1)^{2H-2}\mathrm{d}r_1\mathrm{d}r\\
&\lesssim(t-s)^{\frac{2\kappa}{\alpha}}\sum^\infty_{i=1}\lambda_i^{\kappa-\alpha}q_{i}+\sum^\infty_{i=1}\lambda_i^{-\alpha}q_{i}\int^t_s\int^r_s \left(\lambda_i^{\frac{\alpha}{2}}(r-r_1)\right)^{\frac{2\kappa+\alpha-2H\alpha}{\alpha}}(r-r_1)^{2H-2}\mathrm{d}r_1\mathrm{d}r\\
%&\lesssim(t-s)^{\frac{2\kappa}{\alpha}}\sum^\infty_{i=1}\lambda_i^{\kappa-\frac{\alpha}{2}-H\alpha}q_{i}\\
&\lesssim(t-s)^{\frac{2\kappa}{\alpha}}\sum_{i=1}^{\infty}\left\|A^{\rho}Q^{\frac{1}{2}}\phi_{i}(x)\right\|^2.
\end{split}
\end{equation*}
As $\kappa>H\alpha$, similarly, we have
\begin{equation*}
\begin{split}
&\mathrm{E}\left[\left\|\int^t_sA^{\frac{-\alpha}{2}}\mathscr{S}(t-r)\mathrm{d}B^Q_H(r)\right\|^2\right]\\
&\lesssim(t-s)^{\min\{\frac{2\kappa}{\alpha},2\}}\sum^\infty_{i=1}\lambda_i^{\kappa-\alpha}q_{i}+\sum^\infty_{i=1}\lambda_i^{-\alpha}q_{i}\left(\lambda_i^{\frac{\alpha}{2}}(t-s)\right)^{\min\{\frac{2\kappa}{\alpha},2\}}\int^t_sr^{2H-1}\mathrm{d}r\\
&~~~~+\sum^\infty_{i=1}\lambda_i^{-\alpha}q_{i}\left(\lambda_i^{\frac{\alpha}{2}}(t-s)\right)^{\min\{\frac{2\kappa}{\alpha},2\}}\int^t_s(r-s)^{2H-1}\mathrm{d}r\\
&~~~~+\sum^\infty_{i=1}\lambda_i^{-\alpha}q_{i}\int^t_s\int^r_s \left(\lambda_i^{\frac{\alpha}{2}}(t-r)\right)^{\min\{\frac{2\kappa-2H\alpha}{\alpha},1\}}\left(\lambda_i^{\frac{\alpha}{2}}(r-r_1)\right)(r-r_1)^{2H-2}\mathrm{d}r_1\mathrm{d}r\\
%&\lesssim(t-s)^{\min\{\frac{2\kappa}{\alpha},2\}}\sum^\infty_{i=1}\lambda_i^{\kappa-\alpha}q_{i}+(t-s)^{\min\{\frac{2\kappa}{\alpha},1+2H\}+1}\sum^\infty_{i=1}\lambda_i^{\kappa-\frac{\alpha}{2}-H\alpha}q_{i}\\
&\lesssim(t-s)^{\min\{\frac{2\kappa}{\alpha},2\}}\sum_{i=1}^{\infty}\left\|A^{\rho}Q^{\frac{1}{2}}\phi_{i}(x)\right\|^2.
\end{split}
\end{equation*}
\end{proof}

\begin{remark}
When $B^Q_H(x,t)=\sum^\infty_{i=1}\xi_H^{i}(t)\phi_{i}(x)$, $\alpha+2H\alpha>d$ and $0<\kappa<-\frac{d}{2}+\frac{\alpha}{2}+H\alpha$, using the above steps, we still obtain inequalities \eqref{4eq:3.5-1} and \eqref{4eq:3.5-2}.

\end{remark}

We derive the regularity estimates of \eqref{4eq:1.1} from \eqref{4eq:2.4} and Proposition \ref{4prop:1}.
\begin{theorem}\label{4th:1}
Suppose $\left\|u_0\right\|_{L^2\left(\Omega,\mathbb{H}^\kappa\right)}<\infty$, $\left\|v_0\right\|_{L^2\left(\Omega,\mathbb{H}^{\kappa-\alpha}\right)}<\infty$, $\kappa=\frac{\alpha}{2}+H\alpha+2\rho$ and $0<\kappa<\alpha+\frac{1}{2}$, assumptions \ref{as:2.1} and \ref{as:2.2} hold. Then
\begin{equation*}
\left\|u(t)\right\|_{L^2(\Omega,\mathbb{H}^\kappa)}+\left\|\dot{u}(t)\right\|_{L^2(\Omega,\mathbb{H}^{\kappa-\alpha})}\lesssim 1+\left\|u_0\right\|_{L^2\left(\Omega,\mathbb{H}^\kappa\right)}+\left\|v_0\right\|_{L^2\left(\Omega,\mathbb{H}^{\kappa-\alpha}\right)}
\end{equation*}
and
\begin{equation*}
\begin{split}
&\left\|u(t)-u(s)\right\|_{L^2(\Omega,\mathbb{H}^0)}\\
&\lesssim(t-s)^{\min\left\{\frac{\kappa}{\alpha},1\right\}}\left(\left\|u_0\right\|_{L^2\left(\Omega,\mathbb{H}^\kappa\right)}+\left\|v_0\right\|_{L^2\left(\Omega,\mathbb{H}^{\kappa-\alpha}\right)}+1\right).
\end{split}
\end{equation*}
\end{theorem}

\begin{proof} We observe the regularity of $u(t)$ in the space $L^2(\Omega,\mathbb{H}^\nu)$. First, we obtain the following inequality by using assumption \ref{as:2.2}.
\begin{equation*}
\begin{split}
\mathrm{E}\left[\left\|f\left(u(t)\right)-f\left(0\right)\right\|^2\right]\lesssim\mathrm{E}\left[\left\|u(t)\right\|^2\right],
\end{split}
\end{equation*}
which implies
\begin{equation}\label{4eq:3.8}
\begin{split}
\mathrm{E}\left[\left\|f\left(u(t)\right)\right\|^2\right]\lesssim\mathrm{E}\left[\left\|u(t)\right\|^2\right]+1.
\end{split}
\end{equation}
Using \eqref{4eq:2.4} leads to
\begin{equation*}
\begin{split}
\mathrm{E}\left[\left\|A^{\frac{\kappa}{2}}u(t)\right\|^2\right]&\lesssim\mathrm{E}\left[\left\|A^{\frac{\kappa}{2}}\mathscr{C}(t)u_0\right\|^2\right]+\mathrm{E}\left[\left\|A^{\frac{\kappa-\alpha}{2}}\mathscr{S}(t)v_0\right\|^2\right]\\
&~~~~+\mathrm{E}\left[\left\|\int^t_0A^{\frac{\kappa-\alpha}{2}}\mathscr{S}(t-r)f\left(u(r)\right)\mathrm{d}r\right\|^2\right]\\
&~~~~+\mathrm{E}\left[\left\|\int^t_0A^{\frac{\kappa-\alpha}{2}}\mathscr{S}(t-r)\mathrm{d}B^Q_H(r)\right\|^2\right]\\
&\lesssim\mathrm{E}\left[\left\|A^{\frac{\kappa}{2}}u_0\right\|^2\right]+\mathrm{E}\left[\left\|A^{\frac{\kappa-\alpha}{2}}v_0\right\|^2\right]\\
&~~~~+\int^t_0\mathrm{E}\left[\left\|A^{\frac{\kappa-\alpha}{2}}f\left(u(r)\right)\right\|^2\right]\mathrm{d}r\\
&~~~~+\mathrm{E}\left[\left\|\int^t_0A^{\frac{\kappa-\alpha}{2}}\mathscr{S}(t-r)\mathrm{d}B^Q_H(r)\right\|^2\right]
\end{split}
\end{equation*}
and
\begin{equation*}
\begin{split}
\mathrm{E}\left[\left\|A^{\frac{\kappa-\alpha}{2}}\dot{u}(t)\right\|^2\right]\lesssim&\mathrm{E}\left[\left\|A^{\frac{\kappa}{2}}u_0\right\|^2\right]+\mathrm{E}\left[\left\|A^{\frac{\kappa-\alpha}{2}}v_0\right\|^2\right]+\int^t_0\mathrm{E}\left[\left\|A^{\frac{\kappa-\alpha}{2}}f\left(u(r)\right)\right\|^2\right]\mathrm{d}r\\
&~~~~+\mathrm{E}\left[\left\|\int^t_0A^{\frac{\kappa-\alpha}{2}}\mathscr{C}(t-r)\mathrm{d}B^Q_H(r)\right\|^2\right].
\end{split}
\end{equation*}
Then
\begin{equation*}
\begin{split}
&\mathrm{E}\left[\left\|A^{\frac{\kappa}{2}}u(t)\right\|^2\right]+\mathrm{E}\left[\left\|A^{\frac{\kappa-\alpha}{2}}\dot{u}(t)\right\|^2\right]\\
&\lesssim\mathrm{E}\left[\left\|A^{\frac{\kappa}{2}}u_0\right\|^2\right]+\mathrm{E}\left[\left\|A^{\frac{\kappa-\alpha}{2}}(t)v_0\right\|^2\right]+\int^t_0\mathrm{E}\left[\left\|A^{\frac{\kappa-\alpha}{2}}f\left(u(r)\right)\right\|^2\right]\mathrm{d}r\\
&~~~~+\mathrm{E}\left[\left\|\int^t_0A^{\frac{\kappa-\alpha}{2}}\mathscr{S}(t-r)\mathrm{d}B^Q_H(r)\right\|^2\right]+\mathrm{E}\left[\left\|\int^t_0A^{\frac{\kappa-\alpha}{2}}\mathscr{C}(t-r)\mathrm{d}B^Q_H(r)\right\|^2\right].
\end{split}
\end{equation*}
Proposition \ref{4prop:1} implies
\begin{equation}\label{4eq:3.9}
\begin{split}
&\mathrm{E}\left[\left\|A^{\frac{\kappa}{2}}u(t)\right\|^2\right]+\mathrm{E}\left[\left\|A^{\frac{\kappa-\alpha}{2}}\dot{u}(t)\right\|^2\right]\\
&\lesssim\mathrm{E}\left[\left\|A^{\frac{\kappa}{2}}u_0\right\|^2\right]+\mathrm{E}\left[\left\|A^{\frac{\kappa-\alpha}{2}}(t)v_0\right\|^2\right]+\int^t_0\mathrm{E}\left[\left\|A^{\frac{\kappa-\alpha}{2}}f\left(u(r)\right)\right\|^2\right]\mathrm{d}r+1.
\end{split}
\end{equation}
Second, we need to discuss the estimate of \eqref{4eq:3.9} in different cases. As $0<\kappa\le \alpha$,  we have
\begin{equation}\label{4eq:3.10}
 \mathrm{E}\left[\int^t_0\left\|A^{\frac{\kappa-\alpha}{2}}f\left(u(r)\right)\right\|^2\right]\mathrm{d}r\lesssim\mathrm{E}\left[\int^t_0\left\|f\left(u(r)\right)\right\|^2\right]\mathrm{d}r.
 \end{equation}
Substituting \eqref{4eq:3.8} into \eqref{4eq:3.10}, we have
 \begin{equation}\label{4eq:3.11}
 \begin{split}
 \int^t_0\mathrm{E}\left[\left\|A^{\frac{\kappa-\alpha}{2}}f\left(u(r)\right)\right\|^2\right]\mathrm{d}r&\lesssim\int^t_0\mathrm{E}\left[\left\|u(r)\right\|^2\right]\mathrm{d}r+1\\
 &\lesssim\int^t_0\mathrm{E}\left[\left\|A^{\frac{\kappa}{2}}u(r)\right\|^2\right]\mathrm{d}r+1.
 \end{split}
 \end{equation}
For $\alpha<\kappa<\alpha+\frac{1}{2}$, due to assumption \ref{as:2.2} and \eqref{4eq:2.5}, the following inequality is easily derived.
\begin{equation}\label{4eq:3.12}
\begin{split}
&\int^t_0\mathrm{E}\left[\left\|A^{\frac{\kappa-\alpha}{2}}f\left(u(x,r)\right)\right\|^2\right]\mathrm{d}r\\
&\lesssim\int^t_0\mathrm{E}\left[\left\|f\left(u(x,r)\right)\right\|_{W^{\kappa-\alpha,2}}^2\right]\mathrm{d}r\\
&=\int^t_0\mathrm{E}\left[\left\|f\left(u(x,r)\right)\right\|^2+\int_\Omega\int_\Omega\frac{\left|f\left(u(x,r)\right)-f\left(u(y,r)\right)\right|^2}{\left|x-y\right|^{d+2\kappa-2\alpha}}\mathrm{d}x\mathrm{d}y\right]\mathrm{d}r\\
&\lesssim\int^t_0\mathrm{E}\left[\left\|u(x,r)\right\|^2+\int_\Omega\int_\Omega\frac{\left|u(x,r)-u(y,r)\right|^2}{\left|x-y\right|^{d+2\kappa-2\alpha}}\mathrm{d}x\mathrm{d}y\right]\mathrm{d}r+1\\
&=\int^t_0\mathrm{E}\left[\left\|u(x,r)\right\|^2_{W^{\kappa-\alpha,2}}\right]\mathrm{d}r+1\\
&\lesssim\int^t_0\mathrm{E}\left[\left\|A^{\frac{\kappa-\alpha}{2}}u(x,r)\right\|^2\right]\mathrm{d}r+1\\
&\lesssim\int^t_0\mathrm{E}\left[\left\|A^{\frac{\kappa}{2}}u(x,r)\right\|^2\right]\mathrm{d}r+1.
\end{split}
\end{equation}
By combining \eqref{4eq:3.9}-\eqref{4eq:3.12}, we deduce that
\begin{equation}\label{4eq:3.13}
\begin{split}
&\mathrm{E}\left[\left\|A^{\frac{\kappa}{2}}u(t)\right\|^2\right]+\mathrm{E}\left[\left\|A^{\frac{\kappa-\alpha}{2}}\dot{u}(t)\right\|^2\right]\\
&\lesssim\mathrm{E}\left[\left\|A^{\frac{\kappa}{2}}u_0\right\|^2\right]+\mathrm{E}\left[\left\|A^{\frac{\kappa-\alpha}{2}}(t)v_0\right\|^2\right]+\int^t_0\mathrm{E}\left[\left\|A^{\frac{\kappa}{2}}u(r)\right\|^2\right]\mathrm{d}r+1.
\end{split}
\end{equation}
Then using the Gronwall inequality, the following equation is obtained
\begin{equation}\label{4eq:3.13-1}
\begin{split}
\mathrm{E}\left[\left\|A^{\frac{\kappa}{2}}u(t)\right\|^2\right]&\lesssim\mathrm{E}\left[\left\|A^{\frac{\kappa}{2}}u_0\right\|^2\right]+\mathrm{E}\left[\left\|A^{\frac{\kappa-\alpha}{2}}v_0\right\|^2\right]+1.
\end{split}
\end{equation}
Finally, combining \eqref{4eq:3.13} and \eqref{4eq:3.13-1}, we deduce that
\begin{equation}\label{4eq:3.14}
\begin{split}
&\mathrm{E}\left[\left\|A^{\frac{\kappa}{2}}u(t)\right\|^2\right]+\mathrm{E}\left[\left\|A^{\frac{\kappa-\alpha}{2}}\dot{u}(t)\right\|^2\right]\\
&\lesssim\mathrm{E}\left[\left\|A^{\frac{\kappa}{2}}u_0\right\|^2\right]+\mathrm{E}\left[\left\|A^{\frac{\kappa-\alpha}{2}}(t)v_0\right\|^2\right]+1.
\end{split}
\end{equation}

Now, we present the H\"older regularity estimates of $u(t)$. Using \eqref{4eq:2.4} and Proposition \ref{4prop:1} yields
\begin{equation}\label{4eq:3.15}
\begin{split}
\mathrm{E}\left[\left\|u(t)-u(s)\right\|^2\right]&\lesssim\mathrm{E}\left[\left\|\mathscr{C}(t-s)u(s)-u(s)\right\|^2\right]+\mathrm{E}\left[\left\|A^{-\frac{\alpha}{2}}\mathscr{S}(t-s)\dot{u}(s)\right\|^2\right]\\
&~~~~+\mathrm{E}\left[\left\|\int^t_sA^{-\frac{\alpha}{2}}\mathscr{S}(t-r)f\left(u(r)\right)\mathrm{d}r\right\|^2\right]\\
&~~~~+\mathrm{E}\left[\left\|\int^t_sA^{-\frac{\alpha}{2}}\mathscr{S}(t-r)\mathrm{d}B^Q_H(r)\right\|^2\right]\\
&\lesssim\mathrm{E}\left[\left\|\left(A^{\frac{\alpha}{2}}(t-s)\right)^{\min\{\frac{\kappa}{\alpha},1\}}u(s)\right\|^2\right]\\
&~~~~+\mathrm{E}\left[\left\|A^{-\frac{\alpha}{2}}\left(A^{\frac{\alpha}{2}}(t-s)\right)^{\min\{\frac{\kappa}{\alpha},1\}}\dot{u}(s)\right\|^2\right]\\
&~~~~+\mathrm{E}\left[\left\|\int^t_sf\left(u(r)\right)\mathrm{d}r\right\|^2\right]+(t-s)^{\min\left\{\frac{2\kappa}{\alpha},2\right\}}.
\end{split}
\end{equation}
Employing \eqref{4eq:3.14} and Cauchy-Schwarz-Buniakowsky inequality, we obtain
\begin{equation}\label{4eq:3.16}
\begin{split}
\mathrm{E}\left[\left\|u(t)-u(s)\right\|^2\right]&\lesssim\left(t-s\right)^{\min\{\frac{2\kappa}{\alpha},2\}}\left(\mathrm{E}\left[\left\|A^{\frac{\kappa}{2}}u_0\right\|^2\right]+\mathrm{E}\left[\left\|A^{\frac{\kappa-\alpha}{2}}(t)v_0\right\|^2\right]+1\right)\\
&~~~~+\left(t-s\right)\int^t_s\mathrm{E}\left[\left\|f\left(u(r)\right)\right\|^2\right]\mathrm{d}r+(t-s)^{\min\left\{\frac{2\kappa}{\alpha},2\right\}}\\
&\lesssim\left(t-s\right)^{\min\{\frac{2\kappa}{\alpha},2\}}\left(\mathrm{E}\left[\left\|A^{\frac{\kappa}{2}}u_0\right\|^2\right]+\mathrm{E}\left[\left\|A^{\frac{\kappa-\alpha}{2}}(t)v_0\right\|^2\right]+1\right).
\end{split}
\end{equation}
\end{proof}

\section{Temporal discretization} \label{4sec:4}
We solve \eqref{4eq:1.1} numerically by discretizing \eqref{4eq:2.3}. To begin, we establish a method for approximating stochastic convolution in \eqref{4eq:2.3}. We define each subinterval $\left(t_{j},t_{j+1}\right]$, and set $t_j=j\tau$ for $j=0,1,2,\cdots, \frac{T}{\tau}-1$. Using the stochastic trigonometric method get the following equation
\begin{equation*}
\begin{split}
\left[\begin{split}
u_{j+1}\\
\bar{u}_{j+1}
\end{split} \right]&=\left[\begin{split}
\mathscr{C}\left(\tau\right)~~~~A^{-\frac{\alpha}{2}}\mathscr{S}\left(\tau\right)\\
-A^{\frac{\alpha}{2}}\mathscr{S}\left(\tau\right)~~~~\mathscr{C}\left(\tau\right)
\end{split} \right]\left[\begin{split}u_{j}\\
\bar{u}_{j}
\end{split} \right]+\tau\left[\begin{split}A^{-\frac{\alpha}{2}}\mathscr{S}\left(\tau\right)f\left(u_{j}\right)\\
\mathscr{C}\left(\tau\right)f\left(u_{j}\right)
\end{split} \right]\\
&+\left[\begin{split}&\int^{t_{j+1}}_{t_{j}}A^{-\frac{\alpha}{2}}\mathscr{S}\left(t_{j+1}-r\right)\mathrm{d}B^Q_H(r)\\
&\int^{t_{j+1}}_{t_{j}}\mathscr{C}\left(t_{j+1}-r\right)\mathrm{d}B^Q_H(r)\\
\end{split}  \right].
\end{split}
\end{equation*}
In order to obtain the error analysis of temporal discretization, we will approximate stochastic convolution using the following method.  Again using integration by parts, we get
\begin{equation}\label{4eq:4.1}
\begin{split}
\int^{t_{j+1}}_{t_{j}}A^{-\frac{\alpha}{2}}\mathscr{S}\left(t_{j+1}-r\right)\mathrm{d}B^Q_H(r)&=-A^{-\frac{\alpha}{2}}\mathscr{S}\left(t_{j+1}-t_{j}\right)B^Q_H(t_{j})\\
&~~~~+\int^{t_{j+1}}_{t_{j}}\mathscr{C}\left(A^{\frac{\alpha}{2}}(t_{j+1}-r)\right)B^Q_H(r)\mathrm{d}r.
\end{split}
\end{equation}
Using \eqref{4eq:4.1} yields the approximation of $\int^{t_{j+1}}_{t_{j}}A^{-\frac{\alpha}{2}}\mathscr{S}\left(t_{j+1}-r\right)\mathrm{d}B^Q_H(r)$
\begin{equation*}
\begin{split}
&-A^{-\frac{\alpha}{2}}\mathscr{S}\left(t_{j+1}-t_j\right)B^Q_H(t_{j})+\int^{t_{j+1}}_{t_{j}}\mathscr{C}\left(t_{j+1}-r\right)B^Q_H(t_j)\mathrm{d}r.\\
\end{split}
\end{equation*}
Using a similar method, we approximate $\int^{t_{j+1}}_{t_{j}}\mathscr{C}\left(t_{j+1}-r\right)\mathrm{d}B^Q_H(r)$.
\begin{equation*}
\begin{split}
&B^Q_H(t_{j+1})-\mathscr{C}\left(t_{j+1}-t_{j}\right)B^Q_H(t_{j})-\int^{t_{j+1}}_{t_{j}}A^{\frac{\alpha}{2}}\mathscr{S}\left(t_{j+1}-r\right)B^Q_H(t_j)\mathrm{d}r.\\
\end{split}
\end{equation*}
The above approximation makes numerical simulations of stochastic convolution easy to implement. Then we get a discretization of \eqref{4eq:2.1}
\begin{equation}\label{4eq:4.2}
\begin{split}
\left[\begin{split}
u_{j+1}\\
\bar{u}_{j+1}
\end{split} \right]&=\left[\begin{split}
\mathscr{C}\left(\tau\right)~~~~A^{-\frac{\alpha}{2}}\mathscr{S}\left(\tau\right)\\
-A^{\frac{\alpha}{2}}\mathscr{S}\left(\tau\right)~~~~\mathscr{C}\left(\tau\right)
\end{split} \right]\left[\begin{split}u_{j}\\
\bar{u}_{j}
\end{split} \right]+\tau\left[\begin{split}A^{-\frac{\alpha}{2}}\mathscr{S}\left(\tau\right)f\left(u_{j}\right)\\
\mathscr{C}\left(\tau\right)f\left(u_{j}\right)
\end{split} \right]\\
&+\left[\begin{split}&-A^{-\frac{\alpha}{2}}\mathscr{S}\left(\tau\right)B^Q_H(t_{j})+\int^{t_{j+1}}_{t_{j}}\mathscr{C}\left(t_{j+1}-r\right)B^Q_H(t_j)\mathrm{d}r\\
&B^Q_H(t_{j+1})-\mathscr{C}\left(\tau\right)B^Q_H(t_{j})-\int^{t_{j+1}}_{t_{j}}A^{\frac{\alpha}{2}}\mathscr{S}\left(t_{j+1}-r\right)B^Q_H(t_j)\mathrm{d}r\\
\end{split}  \right].
\end{split}
\end{equation}
The following equation is obtained by applying trigonometric identities.
\begin{equation*}
\left[\begin{split}
\mathscr{C}\left(\tau\right)~~~~A^{-\frac{\alpha}{2}}\mathscr{S}\left(\tau\right)\\
-A^{\frac{\alpha}{2}}\mathscr{S}\left(\tau\right)~~~~\mathscr{C}\left(\tau\right)
\end{split} \right]^n=\left[\begin{split}
\mathscr{C}\left(t_n\right)~~~~A^{-\frac{\alpha}{2}}\mathscr{S}\left(t_n\right)\\
-A^{\frac{\alpha}{2}}\mathscr{S}\left(t_n\right)~~~~\mathscr{C}\left(t_n\right)
\end{split} \right]
\end{equation*}
and
\begin{equation*}
\left[\begin{split}
\mathscr{C}\left(\tau\right)~~~~A^{-\frac{\alpha}{2}}\mathscr{S}\left(\tau\right)\\
-A^{\frac{\alpha}{2}}\mathscr{S}\left(\tau\right)~~~~\mathscr{C}\left(\tau\right)
\end{split} \right]^i\left[\begin{split}A^{-\frac{\alpha}{2}}\mathscr{S}\left(\tau\right)\\
\mathscr{C}\left(\tau\right)
\end{split} \right]=\left[\begin{split}A^{-\frac{\alpha}{2}}\mathscr{S}\left(t_{i+1}\right)\\
\mathscr{C}\left(t_{i+1}\right)
\end{split} \right].
\end{equation*}
Then using recursion form of \eqref{4eq:4.2}, we get
\begin{equation}\label{4eq:4.3}
\begin{split}
u_{n+1}&=\mathscr{C}\left(t_{n+1}\right)u_0+A^{-\frac{\alpha}{2}}\mathscr{S}\left(t_{n+1}\right)v_{0}+\sum^n_{j=0}\tau A^{-\frac{\alpha}{2}}\mathscr{S}\left(A^{\frac{\alpha}{2}} (t_{n+1}-t_j)\right)f\left(u_{j}\right)\\
&~~~~+\sum^n_{j=0}\int^{t_{j+1}}_{t_{j}}\mathscr{C}\left(t_{n+1}-r\right)B^Q_H(t_j)\mathrm{d}r
\end{split}
\end{equation}
and
\begin{equation}\label{4eq:4.3-1}
\begin{split}
\bar{u}_{n+1}&=-A^{\frac{\alpha}{2}}\mathscr{S}\left(t_{n+1}\right)u_0+\mathscr{C}\left(t_{n+1}\right)v_{0}+\sum^n_{j=0}\tau \mathscr{C}\left(A^{\frac{\alpha}{2}} (t_{n+1}-t_j)\right)f\left(u_{j}\right)\\
&~~~~+B^Q_H(t_{n+1})-\sum^n_{j=0}\int^{t_{j+1}}_{t_{j}}A^{\frac{\alpha}{2}}\mathscr{S}\left(t_{n+1}-r\right)B^Q_H(t_j)\mathrm{d}r.
\end{split}
\end{equation}
Let $e_{n+1}=u(t_{n+1})-u_{n+1}$. Combining \eqref{4eq:2.4} and \eqref{4eq:4.3}, we have
\begin{equation}\label{4eq:4.4}
\begin{split}
e_{n+1}&=\sum^n_{j=0}\int^{t_{j+1}}_{t_{j}} A^{-\frac{\alpha}{2}}\left(\mathscr{S}\left(t_{n+1}-s\right)f\left(u(s)\right)-\mathscr{S}\left(t_{n+1}-t_j\right)f\left(u_{j}\right)\right)\mathrm{d}s\\
&~~~~+\sum^n_{j=0}\int^{t_{j+1}}_{t_{j}} \mathscr{C}\left(t_{n+1}-s\right)\left(B^Q_H(s)-B^Q_H(t_j)\right)\mathrm{d}s.
\end{split}
\end{equation}

\begin{proposition}\label{4prop:2}
Let $\kappa=\frac{\alpha}{2}+H\alpha+2\rho$ and $\kappa>0$, assumption \ref{as:2.1} is satisfied, we have
\begin{equation*}
\begin{split}
&\mathrm{E}\left[\left\|\sum^{j=n-1}_{j=0}\int^{t_{j+1}}_{t_{j}}\mathscr{C}\left(t_n-s\right)\left(B^Q_H(s)-B^Q_H(t_j)\right)\mathrm{d}r\right\|^2\right]\\
&\lesssim\left\{\begin{split}\tau^{\frac{2\kappa}{\alpha}},\quad &0<\kappa\le \frac{\alpha}{2}+H\alpha,\\
\tau^{2H+1},\quad &\kappa> \frac{\alpha}{2}+H\alpha.
\end{split}\right.
\end{split}
\end{equation*}
\end{proposition}
\begin{proof}
Thanks to the orthonormal basis $\left\{\phi_{i}(x)\right\}_{i\in\mathbb{N}}$, the following equation holds.
\begin{equation*}
\begin{split}
&\mathrm{E}\left[\left\|\sum^{j=n-1}_{j=0}\int^{t_{j+1}}_{t_{j}}\mathscr{C}\left(t_n-s\right)\left(B^Q_H(s)-B^Q_H(t_j)\right)\mathrm{d}r\right\|^2\right]\\
&=\sum^\infty_{i=1}q_i\sum^{j=n-1}_{j=0}\int^{t_{j+1}}_{t_{j}}\int^{t_{j+1}}_{t_{j}}\cos\left(\lambda_i^{\frac{\alpha}{2}}(t_n-s)\right)\cos\left(\lambda_i^{\frac{\alpha}{2}}(t_n-r)\right)\\
&~~~~\times\mathrm{E}\left[\left(\xi_H^{i}(s)-\xi_H^{i}(t_j)\right)\left(\xi_H^{i}(r)-\xi_H^{i}(t_j)\right)\right]\mathrm{d}s\mathrm{d}r\\
&~~~~+2\sum^\infty_{i=1}q_i\sum^{j=n-1}_{j=1}\sum^{k=j-1}_{k=0}\int^{t_{j+1}}_{t_{j}}\int^{t_{k+1}}_{t_{k}}\cos\left(\lambda_i^{\frac{\alpha}{2}}(t_n-s)\right)\cos\left(\lambda_i^{\frac{\alpha}{2}}(t_n-r)\right)\\
&~~~~\times\mathrm{E}\left[\left(\xi_H^{i}(s)-\xi_H^{i}(t_j)\right)\left(\xi_H^{i}(r)-\xi_H^{i}(t_k)\right)\right]\mathrm{d}s\mathrm{d}r\\
&=J_1+J_2.
\end{split}
\end{equation*}
For $0<\kappa\le H\alpha$, we give the following estimates of $J_1$ and $J_2$ by using the covariance function of FBM.
\begin{equation*}
\begin{split}
J_1&=\sum^\infty_{i=1}q_i\sum^{j=n-1}_{j=0}\int^{t_{j+1}}_{t_{j}}\int^{t_{j+1}}_{t_{j}}\cos\left(\lambda_i^{\frac{\alpha}{2}}(t_n-s)\right)\cos\left(\lambda_i^{\frac{\alpha}{2}}(t_n-r)\right)\\
&~~~~\times\frac{|r-t_j|^{2H}+|s-t_j|^{2H}-|s-r|^{2H}}{2}\mathrm{d}s\mathrm{d}r\\
&=-\sum^\infty_{i=1}q_i\sum^{j=n-1}_{j=0}\int^{t_{j+1}}_{t_{j}}\int^{s}_{t_{j}}\cos\left(\lambda_i^{\frac{\alpha}{2}}(t_n-s)\right)\cos\left(\lambda_i^{\frac{\alpha}{2}}(t_n-r)\right)(s-r)^{2H}\mathrm{d}s\mathrm{d}r\\
&~~~~+\sum^\infty_{i=1}q_i\sum^{j=n-1}_{j=0}\int^{t_{j+1}}_{t_{j}}\int^{t_{j+1}}_{t_{j}}\cos\left(\lambda_i^{\frac{\alpha}{2}}(t_n-s)\right)\cos\left(\lambda_i^{\frac{\alpha}{2}}(t_n-r)\right)(s-t_j)^{2H}\mathrm{d}s\mathrm{d}r\\
&=J_{11}+J_{12}.
\end{split}
\end{equation*}
By integration by parts, we have
\begin{equation*}
\begin{split}
J_{11}&=\sum^\infty_{i=1}\lambda_i^{-\frac{\alpha}{2}}q_i\sum^{j=n-1}_{j=0}\int^{t_{j+1}}_{t_{j}}\cos\left(\lambda_i^{\frac{\alpha}{2}}(t_n-s)\right)\\
&~~~~\times\left(\sin\left(\lambda_i^{\frac{\alpha}{2}}(t_n-r)\right)-\sin\left(\lambda_i^{\frac{\alpha}{2}}(t_n-t_j)\right)\right)(s-r)^{2H}\mathrm{d}s\bigg|^s_{r=t_j}\\
&~~~~+2H\sum^\infty_{i=1}\lambda_i^{-\frac{\alpha}{2}}q_i\sum^{j=n-1}_{j=0}\int^{t_{j+1}}_{t_{j}}\int^{s}_{t_{j}}\cos\left(\lambda_i^{\frac{\alpha}{2}}(t_n-s)\right)\\
&~~~~\times\left(\sin\left(\lambda_i^{\frac{\alpha}{2}}(t_n-r)\right)-\sin\left(\lambda_i^{\frac{\alpha}{2}}(t_n-t_j)\right)\right)(s-r)^{2H-1}\mathrm{d}s\mathrm{d}r\\
&=-2H\sum^\infty_{i=1}\lambda_i^{-\alpha}q_i\sum^{j=n-1}_{j=0}\int^{t_{j+1}}_{t_{j}}\left(\sin\left(\lambda_i^{\frac{\alpha}{2}}(t_n-t_{j+1})\right)-\sin\left(\lambda_i^{\frac{\alpha}{2}}(t_n-r)\right)\right)\\
&~~~~\times\left(\sin\left(\lambda_i^{\frac{\alpha}{2}}(t_n-r)\right)-\sin\left(\lambda_i^{\frac{\alpha}{2}}(t_n-t_j)\right)\right)(t_{j+1}-r)^{2H-1}\mathrm{d}r\\
&~~~~+2H(2H-1)\sum^\infty_{i=1}\lambda_i^{-\alpha}q_i\sum^{j=n-1}_{j=0}\int^{t_{j+1}}_{t_{j}}\int^{s}_{t_{j}}\left(\sin\left(\lambda_i^{\frac{\alpha}{2}}(t_n-s)\right)-\sin\left(\lambda_i^{\frac{\alpha}{2}}(t_n-r)\right)\right)\\
&~~~~\times\left(\sin\left(\lambda_i^{\frac{\alpha}{2}}(t_n-r)\right)-\sin\left(\lambda_i^{\frac{\alpha}{2}}(t_n-t_j)\right)\right)(s-r)^{2H-2}\mathrm{d}s\mathrm{d}r\\
&\lesssim\sum^\infty_{i=1}\lambda_i^{-\alpha}q_i\sum^{j=n-1}_{j=0}\int^{t_{j+1}}_{t_{j}}\left(\lambda_i^{\frac{\alpha}{2}}(t_{j+1}-r)\right)^{\frac{2\kappa+\alpha-2H\alpha}{\alpha}}(t_{j+1}-r)^{2H-1}\mathrm{d}r\\
&~~~~+\sum^\infty_{i=1}\lambda_i^{-\alpha}q_i\sum^{j=n-1}_{j=0}\int^{t_{j+1}}_{t_{j}}\int^{s}_{t_{j}}\left(\lambda_i^{\frac{\alpha}{2}}(s-r)\right)^{\frac{2\kappa+\alpha-2H\alpha}{\alpha}}(s-r)^{2H-2}\mathrm{d}s\mathrm{d}r\\
&\lesssim\tau^{2H-1}\sum^\infty_{i=1}\lambda_i^{-\alpha}q_i\left(\lambda_i^{\frac{\alpha}{2}}\tau\right)^{\frac{2\kappa+\alpha-2H\alpha}{\alpha}}+\sum^\infty_{i=1}\lambda_i^{\kappa-\frac{\alpha}{2}-H\alpha}q_i\sum^{j=n-1}_{j=0}\int^{t_{j+1}}_{t_{j}}\int^{s}_{t_{j}}(s-r)^{\frac{2\kappa}{\alpha}-1}\mathrm{d}s\mathrm{d}r\\
&\lesssim\tau^{\frac{2\kappa}{\alpha}}\left\|A^{\rho}Q^{\frac{1}{2}}\phi_{i}(x)\right\|^2
\end{split}
\end{equation*}
and

\begin{equation*}
\begin{split}
J_{12}&=\sum^\infty_{i=1}q_i\sum^{j=n-1}_{j=0}\int^{t_{j+1}}_{t_{j}}\int^{t_{j+1}}_{t_{j}}\cos\left(\lambda_i^{\frac{\alpha}{2}}(t_n-s)\right)\cos\left(\lambda_i^{\frac{\alpha}{2}}(t_n-r)\right)(s-t_j)^{2H}\mathrm{d}s\mathrm{d}r\\
&=2H\sum^\infty_{i=1}\lambda_i^{-\alpha}q_i\sum^{j=n-1}_{j=0}\int^{t_{j+1}}_{t_{j}}\left(\sin\left(\lambda_i^{\frac{\alpha}{2}}(t_n-s)\right)-\sin\left(\lambda_i^{\frac{\alpha}{2}}(t_n-t_{j+1})\right)\right)(s-t_j)^{2H-1}\mathrm{d}s\\
&~~~~\times\left(\sin\left(\lambda_i^{\frac{\alpha}{2}}(t_n-t_{j})\right)-\sin\left(\lambda_i^{\frac{\alpha}{2}}(t_n-t_{j+1})\right)\right)\\
&\lesssim\sum^\infty_{i=1}\lambda_i^{-\alpha}q_i\sum^{j=n-1}_{j=0}\int^{t_{j+1}}_{t_{j}}\left(\lambda_i^{\frac{\alpha}{2}}(t_{j+1}-s)\right)^{\frac{2\kappa+\alpha-2H\alpha}{\alpha}}(s-t_j)^{2H-1}\mathrm{d}s\\
&\lesssim\tau^{\frac{2\kappa}{\alpha}}\sum^\infty_{i=1}\left\|A^{\rho}Q^{\frac{1}{2}}\phi_{i}(x)\right\|^2.
\end{split}
\end{equation*}
For $J_2$, using the fact $k<j$, we have $t_k\le r\le t_j\le s$. Similar to the estimates of $J_{11}$ and $J_{12}$, we have
\begin{equation*}
\begin{split}
J_2&=\sum^\infty_{i=1}q_i\sum^{j=n-1}_{j=1}\sum^{k=j-1}_{k=0}\int^{t_{j+1}}_{t_{j}}\int^{t_{k+1}}_{t_{k}}\cos\left(\lambda_i^{\frac{\alpha}{2}}(t_n-s)\right)\cos\left(\lambda_i^{\frac{\alpha}{2}}(t_n-r)\right)\\
&~~~~\times\left[-(s-r)^{2H}+(s-t_k)^{2H}+(t_j-r)^{2H}-(t_j-t_k)^{2H}\right]\mathrm{d}s\mathrm{d}r\\
&=\sum^\infty_{i=1}q_i\sum^{j=n-1}_{j=1}\sum^{k=j-1}_{k=0}\int^{t_{j+1}}_{t_{j}}\int^{t_{k+1}}_{t_{k}}\cos\left(\lambda_i^{\frac{\alpha}{2}}(t_n-s)\right)\cos\left(\lambda_i^{\frac{\alpha}{2}}(t_n-r)\right)\\
&~~~~\times\left[2H\int^{r}_{t_{k}}(s-r_1)^{2H-1}\mathrm{d}r_1-2H\int^{r}_{t_{k}}(t_j-r_1)^{2H-1}\mathrm{d}r_1\right]\mathrm{d}s\mathrm{d}r\\
&=2H\sum^\infty_{i=1}q_i\sum^{j=n-1}_{j=1}\sum^{k=j-1}_{k=0}\int^{t_{j+1}}_{t_{j}}\int^{t_{k+1}}_{r_1}\cos\left(\lambda_i^{\frac{\alpha}{2}}(t_n-s)\right)\cos\left(\lambda_i^{\frac{\alpha}{2}}(t_n-r)\right)\\
&~~~~\times\left[\int^{t_{k+1}}_{t_{k}}(s-r_1)^{2H-1}\mathrm{d}r_1-\int^{t_{k+1}}_{t_{k}}(t_j-r_1)^{2H-1}\mathrm{d}r_1\right]\mathrm{d}s\mathrm{d}r\\
&=2H\sum^\infty_{i=1}q_i\lambda_i^{-\frac{\alpha}{2}}\sum^{j=n-1}_{j=1}\sum^{k=j-1}_{k=0}\int^{t_{j+1}}_{t_{j}}\left(\sin\left(\lambda_i^{\frac{\alpha}{2}}(t_n-r_1)\right)-\sin\left(\lambda_i^{\frac{\alpha}{2}}(t_n-t_{k+1})\right)\right)\\
&~~~~\times\cos\left(\lambda_i^{\frac{\alpha}{2}}(t_n-s)\right)\left[\int^{t_{k+1}}_{t_{k}}(s-r_1)^{2H-1}\mathrm{d}r_1-\int^{t_{k+1}}_{t_{k}}(t_j-r_1)^{2H-1}\mathrm{d}r_1\right]\mathrm{d}s\\
&=2H\sum^\infty_{i=1}q_i\lambda_i^{-\alpha}\sum^{j=n-1}_{j=1}\sum^{k=j-1}_{k=0}\int^{t_{j+1}}_{t_{j}}\int^{t_{k+1}}_{t_{k}}\left(\sin\left(\lambda_i^{\frac{\alpha}{2}}(t_n-r_1)\right)-\sin\left(\lambda_i^{\frac{\alpha}{2}}(t_n-t_{k+1})\right)\right)\\
&~~~~\times\left(\sin\left(\lambda_i^{\frac{\alpha}{2}}(t_n-s)\right)-\sin\left(\lambda_i^{\frac{\alpha}{2}}(t_n-t_{j+1})\right)\right)(2H-1)(s-r_1)^{2H-2}\mathrm{d}r_1\mathrm{d}s\\
&\lesssim\tau^{\frac{2\kappa}{\alpha}+1-2H}\sum^\infty_{i=1}\lambda_i^{\kappa-\frac{\alpha}{2}-H\alpha}q_i\sum^{j=n-1}_{j=1}\sum^{k=j-1}_{k=0}\int^{t_{j+1}}_{t_{j}}\int^{t_{k+1}}_{t_{k}}(s-r_1)^{2H-2}\mathrm{d}r_1\mathrm{d}s\\
&=\tau^{\frac{2\kappa}{\alpha}+1-2H}\sum^\infty_{i=1}\lambda_i^{\kappa-\frac{\alpha}{2}-H\alpha}q_i\sum^{j=n-1}_{j=1}\int^{t_{j+1}}_{t_{j}}\int^{t_{j}}_{0}(s-r_1)^{2H-2}\mathrm{d}r_1\mathrm{d}s\\
&\lesssim\tau^{\frac{2\kappa}{\alpha}}\sum^\infty_{i=1}\left\|A^{\rho}Q^{\frac{1}{2}}\phi_{i}(x)\right\|^2.
\end{split}
\end{equation*}
For $\kappa>H\alpha$, similar to the above estimates, we have
\begin{equation*}
\begin{split}
J_1&=-2H\sum^\infty_{i=1}\lambda_i^{-\alpha}q_i\sum^{j=n-1}_{j=0}\int^{t_{j+1}}_{t_{j}}\left(\sin\left(\lambda_i^{\frac{\alpha}{2}}(t_n-t_{j+1})\right)-\sin\left(\lambda_i^{\frac{\alpha}{2}}(t_n-r)\right)\right)\\
&~~~~\times\left(\sin\left(\lambda_i^{\frac{\alpha}{2}}(t_n-r)\right)-\sin\left(\lambda_i^{\frac{\alpha}{2}}(t_n-t_j)\right)\right)(s-r)^{2H-1}\mathrm{d}r\\
&~~~~+2H(2H-1)\sum^\infty_{i=1}\lambda_i^{-\alpha}q_i\sum^{j=n-1}_{j=0}\int^{t_{j+1}}_{t_{j}}\int^{s}_{t_{j}}\left(\sin\left(\lambda_i^{\frac{\alpha}{2}}(t_n-s)\right)-\sin\left(\lambda_i^{\frac{\alpha}{2}}(t_n-r)\right)\right)\\
&~~~~\times\left(\sin\left(\lambda_i^{\frac{\alpha}{2}}(t_n-r)\right)-\sin\left(\lambda_i^{\frac{\alpha}{2}}(t_n-t_j)\right)\right)(s-r)^{2H-2}\mathrm{d}s\mathrm{d}r\\
&~~~~+2H\sum^\infty_{i=1}\lambda_i^{-\alpha}q_i\sum^{j=n-1}_{j=0}\int^{t_{j+1}}_{t_{j}}\left(\sin\left(\lambda_i^{\frac{\alpha}{2}}(t_n-s)\right)-\sin\left(\lambda_i^{\frac{\alpha}{2}}(t_n-t_{j+1})\right)\right)(s-t_j)^{2H-1}\mathrm{d}s\\
&~~~~\times\left(\sin\left(\lambda_i^{\frac{\alpha}{2}}(t_n-t_{j})\right)-\sin\left(\lambda_i^{\frac{\alpha}{2}}(t_n-t_{j+1})\right)\right)\\
&\lesssim\sum^\infty_{i=1}\lambda_i^{-\alpha}q_i\sum^{j=n-1}_{j=0}\int^{t_{j+1}}_{t_{j}}\left(\lambda_i^{\frac{\alpha}{2}}(t_{j+1}-r)\right)^{\min\{\frac{2\kappa-2H\alpha}{\alpha},1\}}\left(\lambda_i^{\frac{\alpha}{2}}(r-t_j)\right)(s-r)^{2H-1}\mathrm{d}r\\
&~~~~+\sum^\infty_{i=1}\lambda_i^{-\alpha}q_i\sum^{j=n-1}_{j=0}\int^{t_{j+1}}_{t_{j}}\int^{s}_{t_{j}}\left(\lambda_i^{\frac{\alpha}{2}}(s-r)\right)\left(\lambda_i^{\frac{\alpha}{2}}(r-t_j)\right)^{\min\{\frac{2\kappa-2H\alpha}{\alpha},1\}}(s-r)^{2H-2}\mathrm{d}s\mathrm{d}r\\
&~~~~+\sum^\infty_{i=1}\lambda_i^{-\alpha}q_i\sum^{j=n-1}_{j=0}\int^{t_{j+1}}_{t_{j}}\left(\lambda_i^{\frac{\alpha}{2}}(t_{j+1}-s)\right)^{\min\{\frac{2\kappa-2H\alpha}{\alpha},1\}}(s-t_j)^{2H-1}\mathrm{d}s\left(\lambda_i^{\frac{\alpha}{2}}\tau\right)\\
&\lesssim\tau^{\min\{\frac{2\kappa}{\alpha},1+2H\}}\sum^\infty_{i=1}\left\|A^{\rho}Q^{\frac{1}{2}}\phi_{i}(x)\right\|^2
\end{split}
\end{equation*}
and
\begin{equation*}
\begin{split}
J_2&=2H\sum^\infty_{i=1}q_i\lambda_i^{-\alpha}\sum^{j=n-1}_{j=1}\sum^{k=j-1}_{k=0}\int^{t_{j+1}}_{t_{j}}\int^{t_{k+1}}_{t_{k}}\left(\sin\left(\lambda_i^{\frac{\alpha}{2}}(t_n-r_1)\right)-\sin\left(\lambda_i^{\frac{\alpha}{2}}(t_n-t_{k+1})\right)\right)\\
&~~~~\times\left(\sin\left(\lambda_i^{\frac{\alpha}{2}}(t_n-s)\right)-\sin\left(\lambda_i^{\frac{\alpha}{2}}(t_n-t_{j+1})\right)\right)(2H-1)(s-r_1)^{2H-2}\mathrm{d}r_1\mathrm{d}s\\
&\lesssim\sum^\infty_{i=1}q_i\lambda_i^{-\alpha}\sum^{j=n-1}_{j=1}\sum^{k=j-1}_{k=0}\int^{t_{j+1}}_{t_{j}}\int^{t_{k+1}}_{t_{k}}\left(\lambda_i^{\frac{\alpha}{2}}(t_{k+1}-r_1)\right)^{\min\{\frac{2\kappa-2H\alpha}{\alpha},1\}}\\
&~~~~\times\left(\lambda_i^{\frac{\alpha}{2}}(t_{j+1}-s)\right)(s-r_1)^{2H-2}\mathrm{d}r_1\mathrm{d}s\\
&\lesssim\tau^{\min\{\frac{2\kappa-2H\alpha}{\alpha},1\}+1}\sum^\infty_{i=1}\lambda_i^{\kappa-\frac{\alpha}{2}-H\alpha}q_i\sum^{j=n-1}_{j=1}\int^{t_{j+1}}_{t_{j}}\int^{t_{j}}_{0}(s-r_1)^{2H-2}\mathrm{d}r_1\mathrm{d}s\\
&\lesssim\tau^{\min\{\frac{2\kappa}{\alpha},1+2H\}}\sum^\infty_{i=1}\left\|A^{\rho}Q^{\frac{1}{2}}\phi_{i}(x)\right\|^2.
\end{split}
\end{equation*}
The proof of Proposition \ref{4prop:2} is completed by combining the preceding estimates.

\end{proof}

Last, let's consider the strong error estimates of the numerical approximation of scheme \eqref{4eq:4.2}. We obtain the following inequalities by using triangle inequality, Theorem \ref{4th:1} and Proposition \ref{4prop:2}.

\begin{theorem}\label{4th:2}
Let $\kappa=\frac{\alpha}{2}+H\alpha+2\rho$, $0<\kappa<\alpha+\frac{1}{2}$, $\left\|u_0\right\|_{L^2\left(\Omega,\mathbb{H}^\kappa\right)}<\infty$, $\left\|v_0\right\|_{L^2\left(\Omega,\mathbb{H}^{\kappa-\alpha}\right)}<\infty$, assumptions \ref{as:2.1}-\ref{as:2.2} hold, and $e_n$ is shown in \eqref{4eq:4.4}, then
\begin{equation*}
\left\|e_n\right\|_{L^2\left(\Omega,\mathbb{H}^0\right)}\lesssim\left\{\begin{split}\tau^{\frac{\kappa}{\alpha}}\left(\left\|u_0\right\|_{L^2\left(\Omega,\mathbb{H}^\kappa\right)}+\left\|v_0\right\|_{L^2\left(\Omega,\mathbb{H}^{\kappa-\alpha}\right)}\right),\quad &0<\kappa\le \frac{\alpha}{2}+H\alpha,\\
\tau^{H+\frac{1}{2}}\left(\left\|u_0\right\|_{L^2\left(\Omega,\mathbb{H}^\kappa\right)}+\left\|v_0\right\|_{L^2\left(\Omega,\mathbb{H}^{\kappa-\alpha}\right)}\right),\quad &\frac{\alpha}{2}+H\alpha <\kappa< \alpha+\frac{1}{2}.
\end{split}\right.
\end{equation*}
\end{theorem}
\begin{proof}
The above estimate follows from  the H\"older regularity estimates of $u(t)$, Proposition \ref{4prop:2}. From \eqref{4eq:4.4}, we have
\begin{equation*}
\begin{split}
\left\|e_{n}\right\|_{L^2\left(\Omega,\mathbb{H}^0\right)}&\lesssim\left\|\sum^{n-1}_{j=0}\int^{t_{j+1}}_{t_{j}} A^{-\frac{\alpha}{2}}\left(\mathscr{S}\left(t_{n}-s\right)f\left(u(s)\right)-\mathscr{S}\left(t_{n}-t_j\right)f\left(u_{j}\right)\right)\mathrm{d}s\right\|_{L^2\left(\Omega,\mathbb{H}^0\right)}\\
&~~~~+\left\|\sum^{n-1}_{j=0}\int^{t_{j+1}}_{t_{j}} \mathscr{C}\left(t_{n}-s\right)\left(B^Q_H(s)-B^Q_H(t_j)\right)\mathrm{d}s\right\|_{L^2\left(\Omega,\mathbb{H}^0\right)}\\
&\lesssim\sum^{n-1}_{j=0}\int^{t_{j+1}}_{t_{j}} \left\|A^{-\frac{\alpha}{2}}\left(\mathscr{S}\left(t_{n}-s\right)f\left(u(s)\right)-\mathscr{S}\left(t_{n}-s\right)f\left(u(t_{j})\right)\right)\right\|_{L^2\left(\Omega,\mathbb{H}^0\right)}\mathrm{d}s\\
&~~~~+\sum^{n-1}_{j=0}\int^{t_{j+1}}_{t_{j}} \left\|A^{-\frac{\alpha}{2}}\left(\mathscr{S}\left(t_{n}-s\right)f\left(u(t_{j})\right)-\mathscr{S}\left(t_{n}-t_j\right)f\left(u(t_{j})\right)\right)\right\|_{L^2\left(\Omega,\mathbb{H}^0\right)}\mathrm{d}s\\
&~~~~+\sum^{n-1}_{j=0}\int^{t_{j+1}}_{t_{j}} \left\|A^{-\frac{\alpha}{2}}\mathscr{S}\left(t_{n}-t_j\right)\left(f\left(u(t_{j})\right)-f\left(u_{j}\right)\right)\right\|_{L^2\left(\Omega,\mathbb{H}^0\right)}\mathrm{d}s\\
&~~~~+\tau^{\min\{\frac{\kappa}{\alpha},H+\frac{1}{2}\}}\\
&\lesssim\sum^{n-1}_{j=0}\int^{t_{j+1}}_{t_{j}} \left\|u(s)-u(t_{j})\right\|_{L^2\left(\Omega,\mathbb{H}^0\right)}\mathrm{d}s+\tau\sum^{n-1}_{j=0}\int^{t_{j+1}}_{t_{j}} \left\|f\left(u(t_{j})\right)\right\|_{L^2\left(\Omega,\mathbb{H}^0\right)}\mathrm{d}s\\
&~~~~+\sum^{n-1}_{j=0}\int^{t_{j+1}}_{t_{j}} \left\|u(t_{j})-u_{j}\right\|_{L^2\left(\Omega,\mathbb{H}^0\right)}\mathrm{d}s+\tau^{\min\{\frac{\kappa}{\alpha},H+\frac{1}{2}\}}\\
&\lesssim
\tau^{\min\left\{\frac{\kappa}{\alpha},1\right\}}\left(\left\|u_0\right\|_{L^2\left(\Omega,\mathbb{H}^\kappa\right)}+\left\|v_0\right\|_{L^2\left(\Omega,\mathbb{H}^{\kappa-\alpha}\right)}+1\right)\\
&~~~~+\tau\left(\left\|u_0\right\|_{L^2\left(\Omega,\mathbb{H}^\kappa\right)}+\left\|v_0\right\|_{L^2\left(\Omega,\mathbb{H}^{\kappa-\alpha}\right)}+1\right)\\
&~~~~+\sum^{n-1}_{j=0}\int^{t_{j+1}}_{t_{j}} \left\|e_{j}\right\|_{L^2\left(\Omega,\mathbb{H}^0\right)}\mathrm{d}s+\tau^{\min\left\{\frac{\kappa}{\alpha},H+\frac{1}{2}\right\}}\\
&\lesssim
\tau^{\min\left\{\frac{\kappa}{\alpha},H+\frac{1}{2}\right\}}\left(\left\|u_0\right\|_{L^2\left(\Omega,\mathbb{H}^\kappa\right)}+\left\|v_0\right\|_{L^2\left(\Omega,\mathbb{H}^{\kappa-\alpha}\right)}+1\right)+\tau\sum^{n-1}_{j=0}\left\|e_{j}\right\|_{L^2\left(\Omega,\mathbb{H}^0\right)}.
\end{split}
\end{equation*}
In the second inequality, we use triangle inequality and Proposition \ref{4prop:2}. According to Theorem \ref{4th:1}, we arrive at the fourth inequality. Then using the discrete Gronwall inequality leads to
\begin{equation*}
\begin{split}
\left\|e_{n}\right\|_{L^2\left(\Omega,\mathbb{H}^0\right)}&\lesssim
\tau^{\min\left\{\frac{\kappa}{\alpha},H+\frac{1}{2}\right\}}\left(\left\|u_0\right\|_{L^2\left(\Omega,\mathbb{H}^\kappa\right)}+\left\|v_0\right\|_{L^2\left(\Omega,\mathbb{H}^{\kappa-\alpha}\right)}+1\right).
\end{split}
\end{equation*}
\end{proof}

\begin{remark}
In fact, the numerical method \eqref{4eq:4.2} is still effective when $H\in[\frac{1}{2},1)$ and $\kappa=\alpha+2\rho$. For $0<\kappa<\alpha$, the proposed scheme \eqref{4eq:4.2} possesses the strong convergence order $\frac{\kappa}{\alpha}$. As $\alpha\le\kappa<\alpha+\frac{1}{2}$, the strong convergence order of the scheme \eqref{4eq:4.2} can reach $1$.
\end{remark}

\section{Numerical Experiments} \label{4sec:5}
Using scheme \eqref{4eq:4.2}, we solve \eqref{4eq:5.1} to illustrate the convergence rates in Theorem \ref{4th:2} and investigate the effect of the parameters $\rho$ and $H$ on the convergence.
\begin{equation}\label{4eq:5.1}
\left\{
\begin{array}{ll}
\mathrm{d} \dot{u}(x,t)+A^\alpha  u(x,t)\mathrm{d}t=\left(\cos(u(x,t))+u(x,t)\right)\mathrm{d}t+\mathrm{d}B^Q_{H}(x,t),& \quad \mathrm{in} \ (0,1)\times(0,T],\\
u(x,0)=\frac{\sqrt{2}}{4}\sin(\pi x),v(x,0)=\frac{\sqrt{2}}{2}\sin(3\pi x),& \quad  \mathrm{in}\ (0,1),\\
u(0,t)=u(1,t)=0.
\end{array}
\right.
\end{equation}
We use the Ckolesky method to generate the trajectories of FBM. Let $u_n$ denote the discrete solution at time $t_n=n \tau$ with fixed time step size $\tau=\frac{T}{N}$. The following formula is used to calculates the strong convergence rates in time:

\begin{eqnarray*}
\textrm{convergence rate}=
 \frac{\ln\left(\sqrt{\mathrm{E}\left[\left\|u_{N}-u_{2N}\right\|^2\right]}\bigg/
\sqrt{\mathrm{E}\left[\left\|u_{2N}-u_{4N}\right\|^2\right]}\right)}{\ln2}.
\end{eqnarray*}
We apply a Monte Carlo method with $2000$ trajectories to approximate the expectation of stochastic process. The space discretization of \eqref{4eq:5.1} follows from the spectral Galerkin method with first $1800$ orthonormal bases, which ensures that the time error is the dominant one.

\begin{table}[H]
\renewcommand\arraystretch{1.6}
\caption{Time convergence rates with $H=0.4$, $T=0.2$ and $\alpha=0.8$}\label{4table:1}
\centering
\begin{tabular}{c c c c c c c c }
\hline
$N$ & $\kappa=0.32$ & Rate &$\kappa=0.52$ &Rate & $\kappa=0.72$ &Rate  \\
\midrule[1.5pt]
  4&0.557e-02&     & 2.301e-02&     &1.204e-02&          \\
  8&0.418e-02& 0.414& 1.457e-02& 0.659&6.668e-03&0.853\\
  16&0.313e-02& 0.417& 9.215e-03& 0.661&3.635e-03&0.875  \\
  \hline
\end{tabular}
\end{table}
First, we consider the case of $0<\kappa\le\frac{\alpha}{2}+H\alpha$. Then the theoretical convergence rate is $\frac{\kappa}{\alpha}$. To observe this convergence order, we choose $N=4,~8,~16,~32$ to approximate solutions $u(x,T)$. Table \ref{4table:1} presents the rates of convergence of the time discretization for varying $\kappa$ with the fixed parameters $\alpha=0.75$ and $H=0.4$. When $\sqrt{q_{i}}=\lambda_i^{-0.05},~\lambda_i^{-0.15}~,\lambda_i^{-0.25}$,  $\kappa$ is $0.32$, $0.52$ and $0.72$, respectively. Then the corresponding theoretical convergence rates are $0.4,~0.65,~0.9$, respectively. Numerical results indicate the strong convergence rate is approximately that given in Theorem \ref{4th:2}.

\begin{table}[H]
\renewcommand\arraystretch{1.6}
\caption{Time convergence rates with $\alpha=0.25$, $T=0.2$ and $\kappa=\frac{\alpha}{2}+H\alpha+0.3$}\label{4table:3}
\centering
\begin{tabular}{c c c c c c c c }
\hline
$N$ & $H=0.1$ & Rate &$H=0.25$ &Rate & $H=0.4$ &Rate  \\
\midrule[1.5pt]
  4&2.004e-02&     & 1.166e-02&     &7.262e-03&          \\
  8&1.311e-02& 0.612& 6.697e-03& 0.800&3.704e-03&0.971\\
  16&8.631e-03& 0.603& 3.882e-03& 0.787&1.950e-03&0.926  \\
  \hline
\end{tabular}
\end{table}
Next, we investigate the convergence order for $\frac{\alpha}{2}+H\alpha<\kappa<\alpha+\frac{1}{2}$. Table \ref{4table:3} displays the temporal convergence rates for three different noise intensity parameters $H=0.1,~0.25,~0.4$ with fixed $\sqrt{q_{i}}=\lambda_i^{-0.4}$. Numerical experiments show that as $H$ decreases, convergence rate increases. This observation completely supports our theoretical result in Theorem \ref{4th:2}.
\section{Conclusion} \label{4sec:6}
In this paper, we developed an efficient time discretization method for solving a SWE forced by FMB with Hurst parameter $H\in(0,\frac{1}{2})$. For $H\in(0,\frac{1}{2})$, we show the strong convergence of this method with uniform time step $\tau$, where the error is of order $O\left(\tau^{H+\min\left\{\frac{\kappa}{\alpha}-\frac{1}{2},\frac{1}{2}\right\}}\right)$. In the case $0<\kappa\le\frac{\alpha}{2}+H\alpha$, the strong convergence rate of the scheme \eqref{4eq:4.2} is equal to H\"older continuity index of $u(t)$. As $\frac{\alpha}{2}+H\alpha<\kappa<\alpha+\frac{1}{2}$, our method strongly converges with order $H+\frac{1}{2}$, which is smaller than H\"older continuity index of $u(t)$, because the strong convergence is limited by the approximation error of stochastic convolution.  In future work, we plan to study the optimal strong convergence order of time discretization for \eqref{4eq:1.1} when $H\in(0,\frac{1}{2})$ and $\frac{\alpha}{2}+H\alpha<\kappa<\alpha+\frac{1}{2}$.
\section*{Acknowledgments}
This work was supported by the Foundation of Hubei Provincial Department of Education (No. B2021255).

%% References with BibTeX database:
%

\bibliographystyle{plain}
\bibliography{myfile}

%% Authors are advised to use a BibTeX database file for their reference list.
%% The provided style file elsarticle-num.bst formats references in the required Procedia style

%% For references without a BibTeX database:

% \begin{thebibliography}{00}

%% \bibitem must have the following form:
%%   \bibitem{key}...
%%

% \bibitem{}

% \end{thebibliography}

\end{document}